\documentclass{amsart}

\usepackage{amssymb, amsmath, amsthm, hyperref, euscript}

\usepackage{amscd}

\usepackage[english]{babel}

\author{Pablo Su\'arez-Serrato and Rafael Torres}

\title[Collapsing, entropy, and Yamabe invariant of symplectic 4-manifolds]{A note on collapse, entropy, and vanishing of the Yamabe invariant of symplectic 4-manifolds}

\address{Instituto de Matem\'aticas - Universidad Nacional Aut\'onoma de M\'exico\\Circuito Exterior, Ciudad Universitaria\\Coyoac\'an, 04510\\Mexico City\\Mexico}

\email{pablo@im.unam.mx}

\address{Scuola Internazionale Superiori di Studi Avanzati\\ Via Bonomea 265\\34136\\Trieste\\Italy}

\email{rtorres@sissa.it}

\theoremstyle{plain}

\newtheorem{theorem}{Theorem}

\newtheorem{corollary}{Corollary}

\newtheorem{proposition}{Proposition}

\theoremstyle{definition}

\newtheorem{definition}{Definition}

\newcommand{\R}{\mathbb{R}}

\newcommand{\Z}{\mathbb{Z}}

\newcommand{\N}{\mathbb{N}}

\newcommand{\C}{\mathbb{C}}

\begin{document}

\begin{abstract} We make use of $\mathcal{F}$-structures and technology developed by Paternain - Petean to compute minimal entropy, minimal volume, and Yamabe invariant of symplectic 4-manifolds, as well as to study their collapse with sectional curvature bounded from below. \`A la Gompf, we show that these invariants vanish on symplectic 4-manifolds that realize any given finitely presented group as their fundamental group. We extend to the symplectic realm a result of LeBrun which relates the Kodaira dimension with the Yamabe invariant of compact complex surfaces. \end{abstract}

\maketitle

\section{Introduction and main results}

Let $(M, g)$ be a closed oriented Riemannian manifold. The minimal entropy $h(M)$ is defined as the infimum of the topological entropy of the geodesic flow of a smooth metric $g$ on $M$. The simplicial volume $||M||$ was defined by Gromov \cite{[G82]} as the infimum of $\Sigma_{i}|r_{i}|$ where $r_{i}$ are the coefficients of any \emph{real} singular cycle representing the fundamental class of $M$. It was introduced in relation to the {\em minimal volume}, which Gromov defined as the infimum of volumes of metrics whose sectional curvature is bounded between -1 and 1. Other types of minimal volumes may be considered, for example, ${\rm Vol}_K(M)$, which is the infimum of volumes of metrics with sectional curvature bounded {\em below} by -1. Due to rescalling properties of these invariants and in order to obtain interesting data, one assumes the normalization $Vol((M, g)) = 1$ of the volume.\\

The computation of these invariants can be challenging. The existence of a circle action, for example, implies that they all vanish. A generalization of a circle action is given by the notion of an $\mathcal{F}$-structure, as introduced by Cheeger and Gromov \cite{[CG1], [G82]} (Definiton \ref{Definition F}). In a series of papers, Cheeger-Gromov \cite{[CG1], [CG2]} and Paternain-Petean \cite{[PP1], [PP2], [PP3]} have shown that the existence of an $\mathcal{F}$-structure implies that $h(M)$, $||M||$, and ${\rm Vol}_K(M)$ are all zero.\\

Recall the minmax definition of the {\em Yamabe invariant} \cite{[Be], [Sc]}. Consider a conformal class \begin{equation}\gamma:= [g] = \{ug : M\overset{u}{\rightarrow} \R^+\}\end{equation} of Riemannian metrics on $M$. The Yamabe constant of $(M, \gamma)$ is defined as:\begin{equation} \mathcal{Y}(M, \gamma):= \underset{g\in \gamma}{\inf} \frac{\int_M {\rm Scal}_g d{\rm vol}_g}{({\rm Vol} (M, g))^{2/n}}.\end{equation}
 
Here ${\rm Scal}_g$ stands for the scalar curvature and $d{\rm vol}_g$ for the volume form that is associated to the smooth metric $g$. The Yamabe invariant of $M$ is then given by: \begin{equation} \mathcal{Y}(M):= \underset{\gamma}{\sup} \mathcal{Y}(M, \gamma). \end{equation} 

LeBrun \cite{[LeB]} and Paternain-Petean \cite{[PP2]} studied the value of these fundamental invariants of Riemannian manifolds on complex surfaces. Given the fundamental role that symplectic manifolds play in our understanding of 4-manifold topology, a goal of the present paper is to study the value of these invariants on the much larger class of symplectic 4-manifolds. The principal technical result is to equip constructions of symplectic 4-manifolds with an $\mathcal{F}$-structure.\\

In order to state our main results, we first recall the following definition \cite{[L1]}. Let $(M, \omega)$ be a minimal symplectic 4-manifold, and let $K_{\omega}$ denote its canonical bundle. The symplectic Kodaira dimension $kod(M, \omega)$ of $(M, \omega)$ is defined as:

\[kod(M, \omega) = \left\{
  \begin{array}{ccccr}
    -\infty & \quad{\rm if}\quad &  K_{\omega} \cdot [\omega]  < 0  & \quad{\rm or}\quad & K_{\omega} \cdot K_{\omega} < 0\\
    0 &  \quad{\rm if}\quad &  K_{\omega} \cdot [\omega] = 0 & \quad{\rm and}\quad &K_{\omega} \cdot K_{\omega} = 0\\
    1 & \quad{\rm if} \quad & K_{\omega} \cdot [\omega] > 0 & \quad{\rm and}\quad & K_{\omega} \cdot K_{\omega} = 0\\
    2 & \quad{\rm if} \quad   & K_{\omega} \cdot [\omega] > 0 & \quad{\rm and}\quad & K_{\omega}\cdot K_{\omega} > 0
  \end{array}
\right.
\]

The symplectic Kodaira dimension of a symplectic manifold is defined as the Kodaira dimension of its minimal model, and in the presence of a holomorphic structure, the holomorphic and the symplectic Kodaira dimensions coincide \cite{[L1], [DZ13]}.

\begin{theorem}{\label{Theorem Gen}} All known examples of closed symplectic 4-manifolds of symplectic Kodaira dimension at most one admit an $\mathcal{F}$-structure.

No closed symplectic 4-manifold of symplectic Kodaira dimension two admits an $\mathcal{F}$-structure.
\end{theorem}

\begin{corollary}{\label{Corollary Gen}} Every known closed symplectic 4-manifold of Kodaira dimension at most one has zero minimal entropy, zero minimal volume, and it collapses with sectional curvature bounded from below. 

Symplectic 4-manifolds of Kodaira dimension two do not collapse with bounded scalar curvature.
\end{corollary}

LeBrun \cite{[LeB]} found an explicit relation between the holomorphic Kodaira dimension and the Yamabe invariant of compact complex surfaces. We extend his result to the symplectic realm. 

\begin{theorem}{\label{Theorem Kod}} Let $(M, \omega)$ be a closed symplectic 4-manifold whose existence is known at the time this paper was written. The following holds.
\[ \mathcal{Y}(M)= \left\{
  \begin{array}{cccr}
     > 0 & \quad {\rm if} \quad & kod(M, \omega) = &  -\infty\\
    0 & \quad {\rm if}\quad & kod(M, \omega) = & 0\\
    0 & \quad {\rm if}\quad  & kod(M, \omega) = & 1\\
    <0 & \quad {\rm if}\quad  & kod(M, \omega) = & 2
  \end{array}
\right.
\]
\end{theorem}

No symplectic 4-manifold of Kodaira dimension one realizes its Yamabe invariant, and neither do several symplectic 4-manifolds of Kodaira dimension zero. This has been previously observed for K\"ahler surfaces by LeBrun \cite{[LeB]}, and on the homeomorphism type of the K3 surface by the second author of this note \cite{[To]}. A result of Hitchin \cite{[Hi]} is fundamental for these purposes.\\

A classical result of Gompf \cite{[Go]} states that any finitely presented group can be the fundamental group of a symplectic 4-manifold. Building on a construction of Baldridge-Kirk \cite{[BK0]}, our third main result addresses the value of these invariants for the manifolds that they constructed in their paper. The precise statement is as follows.

\begin{theorem}{\label{Theorem Prescribed}} Let $G$ be a finitely presented group. There exists a minimal symplectic 4-manifold $M(G)$ with fundamental group $\pi_1(M(G)) \cong G$ and such that $M(G)$ admits an $\mathcal{F}$-structure. Consequently, $$h(M(G)) = 0 = ||M(G)|| = Vol_K(M(G))$$
and $M(G)$ collapses with sectional curvature bounded from below.

Moreover, the Yamabe invariant satisfies \begin{equation}\mathcal{Y}(M(G)) = 0\end{equation} and it is not realized, i.e., there exist no scalar-flat Riemannian metrics on $M(G)$.

\end{theorem}

The organization of the paper is as follows. Section \ref{Section F} contains the definition of our main technical tool and result. Interesting constructions of $\mathcal{F}$-structures are given in Section \ref{Section F}, Section \ref{Section Log}, Section \ref{Section Zero}, and Section \ref{Section One}. The strategy and results used to show the vanishing and non-realization of the Yamabe invariant are described in Section \ref{Section YI}. Theorem \ref{Theorem Prescribed} is proven in Section \ref{Section Proof}. Theorem \ref{Theorem Gen}, Corollary \ref{Corollary Gen}, and Theorem \ref{Theorem Kod} are proven by cases in terms of Kodaira dimensions in Section \ref{Section ProofKod}, where each case is addressed in a subsection.

\subsection{Acknowledgements} R. T. thanks Claudio Arezzo, Steven Frankel, Baptiste, Morin, and Stefano Vidussi for useful conversations. He gratefully acknowledges the Erwin Schr\"odinger International Institute for Mathematical Physics for its support and hospitality during the writing of the manuscript.  PSS thanks CONACyT Mexico and PAPIIT UNAM for supporting various research activities. Both authors would like to thank Jimmy Petean for interesting conversations, and Weiyi Zhang for his interest in our paper.

\section{Collapse, $\mathcal{F}$-structures, and Yamabe invariant}

\subsection{$\mathcal{F}$-structures and collapse}{\label{Section F}} The concept of an $\mathcal{F}$-structure was introduced by Gromov \cite{[G82]}  in terms of sheaves. We will work through out the manuscript with the equivalent definition given in Paternain-Petean \cite[Section 2]{[PP3]}. 

\begin{definition}{\label{Definition F}}An \emph{$\mathcal{F}$-structure} on a smooth closed manifold $M$ is given by 
\begin{itemize}
\item a finite open cover $\{U_1, \ldots, U_N\}$ of $M$;
\item a finite Galois covering $\pi_i: \widetilde{U_i}\rightarrow U_i$ with $\Gamma_i$ a group of deck transformations for $1\leq i \leq N$;
\item a smooth effective torus action with finite kernel of a $k_i$-dimensional torus 
\begin{equation}
\phi_i: T^{k_i}\rightarrow Diff(\widetilde{U_i})
\end{equation}
for $1\leq i \leq N$;
\item a representation $\Phi_i: \Gamma_i \rightarrow Aut(T^{k_i})$ such that
\begin{equation}
\gamma(\phi_i(t)(x)) = \phi_i(\Gamma_i(\gamma)(t))(\gamma x)
\end{equation}
for all $\gamma \in \Gamma_i$, $t\in T^{k_i}$, and $x\in \widetilde{U}_i$;
\item for any subcollection $\{U_{i_1}, \ldots, U_{i_l}\}$ that satisfies $U_{i_1\cdots i_l}:= U_{i_1} \cap \cdots \cap U_{i_l} \neq \emptyset$, the following compatibility condition holds: let $\widetilde{U}_{i_1\cdots i_l}$ be the set of all points $(x_{i_1}, \ldots, x_{i_l}) \in \widetilde{U}_{i_1}\times \cdots \times \widetilde{U}_{i_l}$ such that $\pi_{i_1}(x_{i_1}) = \cdots = \pi_{i_l}(x_{i_l})$. The set $\widetilde{U}_{i_1\cdots i_l}$ covers $\pi^{-1}(U_{i_1\cdots i_l}) \subset \widetilde{U}_{i_1\cdots i_l}$ for all $1\leq j \leq l$. It is required that $\phi_{i_j}$ leaves $\pi^{-1}_{i_j}(U_{i_1\cdots i_l})$ invariant, and it lifts to an action on $\widetilde{U}_{i_1\cdots i_l}$ such that all lifted actions commute.
\item An $\mathcal{F}$-structure is called a \emph{$\mathcal{T}$-structure} if the Galois coverings $\pi_i: \widetilde{U}_i \rightarrow U_i$  can be taken to be trivial for every $i$.
\end{itemize}
\end{definition}

The reader is directed towards \cite{[CG1], [CG2], [PP1], [PP2], [PP3], [SS09], [PSS], [To]} for interesting constructions of $\mathcal{T}$-structures. The key ingredient that we use to prove our main results is the following theorem, which explains a motivation to equip manifolds with a $\mathcal{F}$-structure.

\begin{theorem}{\label{Theorem PP}} Paternain-Petean \cite{[PP1]}. If a closed n-manifold $M$ admits an $\mathcal{F}$-structure, then \begin{equation} h(M) = 0 = ||M|| = Vol_K(M), \end{equation} and it collapses with sectional curvature bounded from below.

If $n\geq 3$, then the Yamabe invariant of $M$ satisfies \begin{equation}\mathcal{Y}(M) \geq 0.\end{equation}
\end{theorem}

Every $\mathcal{F}$-structure that is constructed in this manuscript is a $\mathcal{T}$-structure.

\subsection{Logarithmic transformations and polarized $\mathcal{T}$-structures}{\label{Section Log}} Let $T$ be a 2-torus of self-intersection zero that is contained inside a 4-manifold $M$. The tubular neighborhood $\nu(T)$ is diffeomorphic to $T^2\times D^2$. A logarithmic transfomation is the procedure of deleting the submanifold $\nu (T)$ from $M$ and gluing back $T^2\times D^2$  \begin{equation}(M - \nu(T)) \cup_{\varphi} (T^2\times D^2)\end{equation} using a diffeomorphism $\varphi: T^2\times \partial D^2$ to identify the common boundaries (see \cite{[BK0], [CL]} for more details). In what follows, $\Sigma_g$ denotes a closed oriented surface of genus $g\in \N$.

\begin{proposition}{\label{Proposition Luttinger}} There exists a 4-manifold $M_g(i)$ that is obtained through an application of $i\in \{0, 1, \ldots, 2g\}$ logarithmic transformations to $T^2\times \Sigma_g$ along homologically essential Lagrangian tori taken with respect to the symplectic form  $\pi^{\ast}\omega_{T^2}\oplus \pi^{\ast}\omega_{\Sigma_g}$, which admits a polarized $\mathcal{T}$-structure for all choices of $i$ and $g\in \N$. Moreover,  $b_1(M_g(i)) = 2 + 2g - i$, and $c_2(M_g(i)) = 0 = \sigma(M_g(i))$.
\end{proposition}

\begin{proof} The Lagrangian tori that are used in the logarithmic transformations are of the form \begin{equation}\{x_1\}\times S^1\times S^1\subset S^1\times (S^1\times \Sigma_g) = T^2\times \Sigma_g\end{equation} and \begin{equation}S^1\times \{x_2\}\ \times S^1 \subset S^1\times S^1\times \Sigma_g = T^2 \times \Sigma_g.\end{equation} Let $\alpha_{j}$ be a curve in $T^2$ that carries a generator of the group $\pi_1(T^2)$, and let $\beta_{k}$ be a curve in $\Sigma_g$ that carries a generator of $\pi_1(\Sigma_g)$. We denote the Lagrangian tori by $T_{j, k} :=\alpha_{j}\times \beta_{k}$. When handling several tori simultaneously, we will use the heavy notation $\alpha^{(i)}_j, \beta^{(i)}_k$, and $T^{(i)}_{j, k}$. The curves $\alpha^{(i)}_j$ are parallel push offs of $\alpha_j$, and the curves $\beta^{(i)}_k$ of $\beta_k$ (see \cite{[BK]} for details). The Lagrangian framings will be used for these curves on every logarithmic transformation, and we denote them by $S^1_{\alpha_i}$ and $S^1_{\beta_k}$. The meridian $\mu_{j, k}$ of $T_{j, k}$ in the complement of the tubular neighborhood of the torus inside $M$ is a curve within the same isotopy class of $\{t\}\times \partial D^2 \subset \partial \nu(T_{j, k})$. Recall that two homotopic loops inside a 4-manifolds are isotopic. The diffeomorphism $\varphi$ used to glue the pieces together satisfies $\varphi_{\ast}([\partial D^2]) = p [S^1_{\alpha_j}] + q[S^1_{\beta_k}] + r [\mu_{j, k}]$ in the homology group $H_1(M - \nu(T_{j, k}); \Z) = \Z^3$. Without loss of generality, we can assume that the tori $T_{j, k}$ are disjoint and the surgeries that will be performed require only one of the integers $p$ or $q$ to be non-zero.

Each such such logarithmic transformations reduces the first Betti number by one, and its second Betti number by two. The Euler characteristic is invariant under torus surgeries, and Novikov additivity implies that so is the signature. In particular, we have $c_2(M_g(i)) = c_2(T^2\times \Sigma_g) = 0$, and $\sigma(M_g(i)) = \sigma(T^2\times \Sigma_g) = 0$.

Fix arbitrary $g$ and $i$; these choices depend on the desired values of first Betti number of $M_g(i)$. We apply the logarithmic transformations simultaneously to $T^2\times \Sigma_g$ to obtain \begin{equation} M_g(i):= (T^2\times \Sigma_g - \underset{i}{\bigsqcup} \nu(T_{j, k})) \bigcup_{\varphi_{i}} (\underset{i}{\bigsqcup} T^2 \times D^2),\end{equation} where notation is being abused. Up to diffeomorphism of the resulting manifold, the gluing diffeomorphism $\varphi_{i}$ can be replaced by an affine transformation $A_{i}$ of the common $T^3$ boundary that is isotopic to $\varphi_{i}$. 

The task at hand is to equip $T^2\times \Sigma_g$ and each copy of $T^2\times D^2$ with polarized $\mathcal{T}$-structures such that they commute with each other at their common $T^3$-boundaries. We will do so by equipping each piece in the decomposition of $M_g(i)$ with a free circle action. This yields a polarized $\mathcal{T}$-structure on $M_g(i)$, and we proceed to do so now. For each curve $\alpha^{(i)}_{j}$ there exists a free $S^1$-action $\sigma_{j}$ on $T^2\times \Sigma_g$ whose orbits are in the same homotopy class of $\alpha^{(i)}_{j}$, and these actions commute among them.

Now equip each copy of $T^2 \times D^2$ with a fixed point free $S^1$-action $\tau_{i}$ whose orbits are in the homotopy class of the image of $\alpha^{(i)}_j$ under the affine transformation $A_{i}\alpha^{(i)}_{j}$. Since the actions $\sigma_{i}$ and $\tau_{i}$ commute with respect to conjugation with the affine maps $A_{i}$\begin{equation}A_{i}^{-1}\tau_{i} A_{i} \sigma_{i} = \sigma_{i} A_{i}^{-1}\tau_{i} A_{i},\end{equation} (see \cite{[SS09]} for details). The local circle actions paste together to yield a polarized $\mathcal{T}$-structure on $M_g(i)$. Since the choices of values of $i$ and $g$ were arbitrary, this concludes the proof. \end{proof}

\subsection{Vanishing of Yamabe invariant}{\label{Section YI}}  In this section we recall the results that we use to prove the vanishing of the Yamabe invariant for symplectic 4-manifolds of nonnegative Kodaira dimension, and its non-realization. In particular, we build greatly upon work of Hitchin, Kazdan-Wagner, LeBrun, Taubes, and Paternain-Petean. Theorem \ref{Theorem PP} states that in the presence of an $\mathcal{F}$-structure, there is the collapse $Vol_K(M) = 0$, which implies $Vol_{Ric}(M) = 0 = Vol_{|Scal|}(M)$. The sectional curvature is denoted by $K$ and the Ricci curvature by Ric. If the dimension of $M$ is at least three, then $Vol_{|Scal|}(M) = Vol_{Scal}(M)$, and $Vol_{Scal}(M) = 0$ is equivalent to $\mathcal{Y}(M) \geq 0$ \cite[Proposition 5]{[LeB]}.\\

The non-vanishing of the Seiberg-Witten invariant of a symplectic 4-manifold \cite{[Ta]} is a known obstruction for the existence of a Riemannian metric of positive scalar curvature \cite{[W94]}. Moreover, $\mathcal{Y}(M) > 0$ if and only if there exists a Riemannian metric of positive scalar curvature on $M$. Hence, one obtains the first part of the following known result.

\begin{proposition}{\label{Proposition ZeroYamabe}} Let $M$ be a closed manifold of dimension greater than two, which collapses with scalar curvature bounded from below. If $M$ does not admit a Riemannian metric of positive scalar curvature, then its Yamabe invariant satisfies \begin{equation} \mathcal{Y}(M) = 0,\end{equation} and any scalar-flat Riemannian metric on $M$ is Ricci-flat.

\end{proposition}

The statement concerning the vanishing of the Ricci curvature is an instance of Kazdan-Wagner's result on the trichotomy problem of scalar curvature of a Riemannian metric on a compact manifold \cite[4.35 Theorem]{[Be]}.\\

To prove that the Yamabe invariant is not realized, we will use a beautiful theorem of Hitchin \cite{[Hit1]}. In the case of Kodaira dimension zero and one, we have $c_1^2(M) = 2c_2(M) + 3\sigma(M) = 0$. Hitchin's result \cite[6.37 Theorem]{[Be]} states that a Ricci-flat oriented 4-manifold has either zero sectional curvature, is the K3 surface, the Enriques surface, or a quotient of the later by a free antiholomorphic involution. Since the holomorphic Kodaira dimension coincides with the symplectic one, this immediately implies that no symplectic 4-manifold of Kodaira dimension one can realize a vanishing Yamabe invariant, and restricts those of Kodaira dimension zero that do realize it (see \cite{[LeB], [To]}).

\subsection{Symplectic sums: proof of Theorem \ref{Theorem Prescribed}}{\label{Section Proof}} Let $M_1$ and $M_2$ be closed symplectic 4-manifolds, each containing an embedded symplectic torus $T_{M_1}\subset M_1$ and $T_{M_2}\subset M_2$ with trivial normal bundles. In particular their tubular neighborhoods $\nu(T_{M_1})$ and $\nu(T_{M_2})$ are diffeomorphic to $T^2\times D^2$. The \emph{symplectic sum}  of $M_1$ and $M_2$ along  $T_{M_1}$ and $T_{M_2}$ is \begin{equation}M_1\#_{T_{M_1} = T_{M_2}} M_2 := (M_1 - \nu(T_{M_1})) \cup_{\varphi} (M_2 - \nu(T_{M_2}))\end{equation} where the gluing map $\varphi: \partial \nu(T_{M_1})\rightarrow \partial \nu(T_{M_2}) = T^2\times \partial D^2 = T^3$ is an orientation-reversing diffeomorphism. Moreover, any diffeomorphism of the 3-torus is isotopic to an affine transformation (see \cite[Corollary 10]{[SS09]} for a proof).\\

We now proceed to prove Theorem \ref{Theorem Prescribed}. We build upon the following result.

\begin{theorem}{\label{Theorem BK}} Baldridge-Kirk \cite{[BK0]}. Let $G$ be a group with a presentation that consists of $g$ generators and $r$ relations. There exists a minimal symplectic 4-manifold $M(G)$ with fundamental group $\pi_1(M(G)) = G$, and characteristic numbers $c_2(M(G)) = 12(g + r + 1)$ and $\sigma(M(G)) = - 8(g + r + 1)$.
\end{theorem}

The first claim of Theorem \ref{Theorem Prescribed} is that the symplectic manifold $M(G)$ admits a $\mathcal{T}$-structure, and we proceed to construct it. Baldridge-Kirk built the manifold $M(G)$ as a symplectic sum of $g + r + 1$ copies of the elliptic surface $E(1)$ with a symplectic manifold of the form $X = Y\times S^1$, where $Y$ is a fibered 3-manifold. The group $G$ can be expressed in terms of the fundamental group of $X$ as \begin{equation} G\cong \pi_1(X)/N(s, t, \gamma_1, \cdots, \gamma_{r +g}),\end{equation}contains classes where $\{s, t, \gamma_1, \cdots, \gamma_{r + g}\}$ are classes in the fundamental group of $X$ and $N(s, t, \gamma_1, \cdots, \gamma_{r+g})$ is the normal subgroup generated by such classes. The pieces are pasted together along regular fibers of the elliptic surfaces, and symplectic tori $\{T_0, T_1, \cdots, T_{g+r}\}$ contained in $X$. The two generators of $\pi_1(T_0)$ represent the class $s$ and $t$, and the two generators of $\pi_1(T_i)$ represent the classes $s$ and $\gamma_i$. All these tori have self-intersection zero, and their tubular neighborhoods are diffeomorphic to $T^2\times D^2$ (see \cite{[Go], [BK0]} for further details on the construction). The task at hand is to choose $\mathcal{F}$-structures on the elliptic surface and on the product 4-manifold, which commute with the gluing map at the common $T^3$-boundaries. This ensures that they paste well together into a global $\mathcal{F}$-structure on $M(G)$. In particular, all the $\mathcal{F}$-structures involved are $\mathcal{T}$-structures.\\

The choices of $\mathcal{T}$-structures are the following. Take the canonical $S^1$-action on $X = Y\times S^1$ given by rotation on the second factor, and denote this polarized $\mathcal{T}$-structure by $\tau_0$. The $\mathcal{T}$-structure we use for $E(1)$ blocks was constructed by Paternain-Petean in \cite[Proof of Theorem 5.10]{[PP1]} using its structure of an elliptic fibration, and we proceed to describe it. Consider the diffeomorphism $J := (I, H): T^2\times S^2\rightarrow T^2\times S^2$ with eight fixed points given by $I:\R^2/\Z^2\rightarrow \R^2/\Z^2 = T^2$ as $z\mapsto - z$, and $H:S^2\rightarrow S^2$ as the 180 degrees rotation about the z-axis. Let $B\subset \C^2$ be a ball on which  the circle acts by $\lambda\cdot (w_1, w_2) = (w_1, \lambda\cdot w_2)$. Consider $U:= \{(z, l)\in B\times \mathbb{CP}^1 : z\in l\}$, and the canonical projection $U\rightarrow B$. The aforementioned circle action on the ball $B$ commutes with the involution $J$, and hence it induces a circle action on $U$. Inside $T^2\times S^2$, identify a neighborhood of each fixed point with $B$ and replace it with $U$ to construct a complex surface $\tilde{S}$. The involution $J$ extends to an involution $\tilde{J}: \tilde{S}\rightarrow \tilde{S}$ whose fixed point set consists of eight spheres. In particular, one has $E(1) = \tilde{S}/\tilde{J}$ and it can be expressed as $E(1)\rightarrow S^2$ with $T^2$ as a generic fiber. Paternain-Petean extend the circle action on the $U$ piece to a circle action on the fibers corresponding to the pre images of the north and south poles of $E(1)\rightarrow S^2$. On the complement of these fibers, the elliptic surface is the total space of a fiber bundle with structure group $\{id, I\}$. Call this $\mathcal{T}$-structure $\tau_1$.\\

The gluing diffeomorphism $\varphi$ that is used to construct the symplectic sum of $X$ and a copy of $E(1)$ can be assumed to be isotopic to an affine transformation of the common 3-torus boundary of the building blocks. With respect to this affine representative of the isotopy class of $\varphi$, the circle actions $\tau_0$ and $\tau_1$ commute after conjugation with the affine map. Therefore, they paste well together to a global $\mathcal{T}$-structure (the precise computations can be found in \cite{[SS09]}). Moreover, this extension property for $\mathcal{T}$-structures holds true for a finite arbitrary number of symplectic sums. Therefore, equipping each of the $g + r + 1$ copies of  $E(1)$ with the $\tau_1$ structure and $X$ with the circle action $\tau_0$ yield a $\mathcal{T}$-structure on the symplectic manifold $M(G)$. \\

Theorem \ref{Theorem PP} says its minimal topological entropy and simplicial volume vanish, and that it collapses with sectional curvature bounded from below. Proposition \ref{Proposition ZeroYamabe} now says that its Yamabe invariant is $\mathcal{Y}(M(G)) = 0$. Indeed, Taubes' results \cite{[Ta]} imply that the Seiberg-Witten invariants of $M(G)$ are non-trivial, and therefore it does not admit a Riemannian metric of positive scalar curvature. Suppose the Yamabe invariant were achieved, that is, that there exists a Riemannian metric of zero scalar curvature on $M(G)$. Such a metric would have to be Ricci-flat (cf. Proposition \ref{Proposition ZeroYamabe}), and in particular an Einstein metric \cite{[Be]}. The characteristic numbers of Theorem \ref{Theorem BK} imply that the equality is achieved $|\sigma(M(G))| = \frac{3}{2} c_2(M(G))$. A well-known result of Hitchin \cite[6.37 Theorem]{[Be]} implies that the holomorphic Kodaira dimension of such a manifold is zero. Recall that the symplectic Kodaira dimension of $M(G)$ is one. Since the holomorphic Kodaira dimension, and the symplectic Kodaira dimension coincide, we conclude this is a contradiction. In particular, the Yamabe invariant of $M(G)$ vanishes and it is not realized (cf. Section \ref{Section YI}). This concludes the proof of Theorem \ref{Theorem Prescribed}. 

\section{Proof of Theorem \ref{Theorem Gen}, Corollary \ref{Corollary Gen}, and Theorem \ref{Theorem Kod}}{\label{Section ProofKod}}

As it was mentioned in the introduction, the proofs of our main results are done case-by-case in terms of the symplectic Kodaira dimension. Each dimension is addressed in a subsection. Notice that it suffices to consider only minimal symplectic 4-manifolds in our proofs. Indeed, any non-minimal symplectic 4-manifold $N$ is diffeomorphic to a connected sum $N = M\# k \overline{\mathbb{CP}^2}$ for $k\in \N$, where $M$ is its symplectic minimal model. The manifold $\overline{\mathbb{CP}^2}$ admits a non-trivial $S^1$-action, i.e., a $\mathcal{T}$-structure. Paternain-Petean \cite[Theorem 5.9]{[PP1]} have shown that the existence of $\mathcal{F}$-structures is closed under connected sums.

\subsection{Kodaira dimension $- \infty $ } Liu \cite{[L96]} has classified minimal symplectic 4-manifolds of Kodaira dimension $- \infty$. Such a manifold is either rational or ruled. A rational symplectic 4-manifold is diffeomorphic to $S^2\times S^2$ or to $\mathbb{CP}^2$ blown up $k$ times. A ruled symplectic 4-manifold is diffeomorphic to an $S^2$-bundle over a Riemann surface blown up $k$ times. In particular, all these manifolds admit K\"ahler structure. Their collapse was studied by LeBrun in \cite{[LeB]}. Paternain-Petean \cite{[PP1], [PP2]} have equipped these manifolds with a $\mathcal{T}$-structure, yielding the following result.

\begin{proposition} LeBrun, Liu, Paternain-Petean. Every symplectic 4-manifold of Kodaira dimension $- \infty$ admits an $\mathcal{F}$-structure, has vanishing minimal entropy and simplicial volume, and collapses with sectional curvature bounded from below. Its Yamabe invariant is positive. 
\end{proposition}

\subsection{Kodaira dimension zero}{\label{Section Zero}} The known examples of symplectic 4-manifolds of vanishing Kodaira dimension at the time of writing of this note can be summarized in the following list.

\begin{itemize}
\item K3 surface, 
\item Enriques surface,
\item $T^2$-bundles over $T^2$,
\item $S^1$-bundle over $Y$, where $Y$ is a 3-manifold that fibers over the circle,
\item primary Kodaira surface, and
\item cohomologically symplectic infrasolvmanifolds.

\end{itemize}

The reader is directed towards \cite{[L1], [L2], [FV13]} for further details. In particular, all known symplectic 4-manifolds of zero Kodaira dimension with positive first Betti number are infrasolvmanifolds (see \cite{[Hi]} for the definition).

\begin{proposition}{\label{Kodaira Zero}} Let $M$ be a minimal symplectic closed 4-manifold of symplectic Kodaira dimension zero that is included in the previous list. Then $M$ admits an $\mathcal{F}$-structure, and consequently 
\begin{equation} h(M) = 0 = ||M|| = Vol_K(M),\end{equation} and $M$ collapses with sectional curvature bounded from below. 

The Yamabe invariant satisfies \begin{equation} \mathcal{Y}(M) = 0, \end{equation} and it is realized if and only if $M$ is diffeomorphic to the K3 surface, Enriques surface, a complex torus or a hyperelliptic complex surface.

Moreover, if $c_2(M) = 0$, the $\mathcal{F}$-structure is polarized and \begin{equation} MinVol(M) = 0,\end{equation} and $M$ collapses with bounded sectional curvature. 
\end{proposition}

\begin{proof} The two known examples with small fundamental group are the K3 surface and the Enriques surface. They are both complex elliptic surfaces, hence they admit an $\mathcal{F}$-structure \cite[Theorem 5.10]{[PP1]}. The manifolds in the list that have positive first Betti number are infrasolvmanifolds \cite[Chapter 8]{[Hi]}, and a polarized $\mathcal{F}$-structure was constructed on them in \cite{[SS09]}. The claim $MinVol(M) = 0$ follows from \cite{[CG1]}. We add that this claim also follows from Proposition \ref{Proposition Luttinger}. Let $G$ be the fundamental group of a minimal symplectic 4-manifold $M$ of symplectic Kodaira dimension zero, and suppose $b_1(G) \neq 0$. The work of Bauer \cite{[Ba]} and Li \cite{[L3]} implies $2\leq b_1(G) \leq 4$ and $c_2(M) = 0 = \sigma(M)$. Examples with these Betti and characteristic numbers can be constructed by performing Luttinger surgeries \cite{[Lu], [ADK]} to the 4-torus as in \cite{[BK]}. It was proven in \cite{[CL]} that this cut-and-paste operation preserves Kodaira dimension. Proposition \ref{Proposition Luttinger} implies that these manifolds admit a polarized $\mathcal{F}$-structure.\end{proof}

\subsection{Kodaira dimension one}{\label{Section One}} As it was mentioned in the introduction, Gompf's work \cite{[Go]} implies that symplectic 4-manifolds of positive kodaira dimension are not classifiable. Focus is hence shifted to illustrate their diversity in terms of the geography problem \cite{[Go]}. Baldridge-Li \cite{[BL]} have studied the geography of symplectic 4-manifolds with Kodaira dimension one, and our task in this section is to equipped the manifolds that they have constructed with a $\mathcal{T}$-structure. Notice that the symplectic manifolds of Theorem \ref{Theorem Prescribed} have Kodaira dimension one. We begin by recalling the definitions that are needed to state the result.\\

The degeneracy of a symplectic 4-manifold $(M, \omega)$ is the rank of the kernel of the map \begin{equation} \cup[\omega]: H^1(M; \R)\rightarrow H^3(M; \R).\end{equation}

\begin{definition} A triple $(a, b, c)\in \Z^3$ is said to be admissible if and only if $a = 8k$ for a non-positive integer $k$, \begin{equation} b\geq max\{0, 2 + a/4\}\end{equation} and \begin{equation} 0\leq c \leq b,\end{equation} and $b - c$ is an even number.
\end{definition} The integer $a$ is the signature $\sigma$ of the 4-manifold $(M, \omega)$, $b$ is the first Betti number, and $c$ its degeneracy. 

\begin{proposition}{\label{Proposition BL}} The minimal symplectic 4-manifold $(M, \omega)$ of Kodaira dimension one that realizes any admissible triple $(a, b, c)$ in the sense that \begin{equation}(a, b, c) = (\sigma(M), b_1(M), d(M, \omega))\end{equation} admits a $\mathcal{T}$-structure. Consequently, every such symplectic 4-manifold has \begin{equation} h(M) = 0 = ||M|| = Vol_K(M),\end{equation} and $M$ collapses with sectional curvature bounded from below. 

Moreover, the Yamabe invariant satisfies \begin{equation} \mathcal{Y}(M) = 0\end{equation} and it is not realized, i.e., there exist no scalar-flat Riemannian metrics on $M$.
\end{proposition}

We  recall now Baldridge-Li's \cite{[BL]} construction of these symplectic manifolds. Consider a surface $\Sigma$ of genus $g>0$, a diffeomorphism $\phi$  of $\Sigma$, and the mapping torus $Y=(\Sigma \times [0,1]) / ((x,1) \sim (\phi(x), 0))$.

In their work, Baldridge-Li use certain symplectic $S^1$-bundles over $Y$ that they call {\em bundle manifolds}. The symplectic manifolds that realize any admissible triple in Proposition \ref{Proposition BL} are symplectic sums of bundle manifolds and simply connected elliptic surfaces. For simplicity we will focus on the aspect of their construction that yields a $\mathcal{T}$-structure, and the rest of our claimed results.

\begin{proof} Any circle bundle over a 3-manifold $Y$ admits a free circle action, and hence a $\mathcal{T}$-structure. Let $t$ be a section of the projection $\pi : Y\to S^1$. The preimage $T:= \pi^{-1}(t)$ is a non-trivial torus contained inside $Y$. Call $s$ the loop spanned by the $S^1$-factor of $Y\times S^1$, and express the torus as $T=s \times t$. Define $\sigma$ as the obvious circle action on the $S^1$-factor of $Y\times S^1$. 

Let $T'$ be a regular fiber of $E(n)$ in the neighborhood of a cusp fibre. Take a tubular neighborhood $\nu(T')$ for the elliptic fibration about $T'$, so that $\nu(T')=T^2\times D^2$. Define, for $\tau_{i}, \theta_{i}$ in $S^1$ and $z$ in $D^2$:
$$\tau :T^2\times T^2 \times D^2 \to T^2 \times D^2 \quad ; \quad \tau(\tau_1, \tau_2, \theta_1, \theta_2, z)= (\tau_1\theta_1, \tau_2\theta_2, z)$$

Observe that $E(n)$ admits a $\mathcal{T}$-structure which includes the action $\tau$ on  $\nu(T')$.

Next let $M = E(n)\#_{T=T'}(Y\times S^1)$ denote the symplectic sum \cite{[Go]} of $E(n)$ and $Y\times S^1$ along $T$ and $T'$. Then, the $\mathcal{T}$-structure on $E(n)$ and the circle action $\sigma$ on  $Y\times S^1$ together form a $\mathcal{T}$-structure on M (as in the proof of Theorem \ref{Theorem Prescribed}). The gluing map used in the symplectic sum construction can be replaced by an affine transformation $A$ of $T^3$ in the same isotopy class. The actions $\sigma$ and $\tau$ commute with respect to conjugation with $A$, i.e., $A^{-1}\tau A \sigma = \sigma A^{-1}\tau A$ \cite{[SS09]}. Therefore, we have equipped $M$ with a $\mathcal{T}$-structure. 

Theorem \ref{Theorem PP} implies that $h(M) =  0 = ||M|| = Vol_K(M)$ and $M$ collapses with curvature bounded below. The claims that concern the Yamabe invariant are proven using an argument verbatim to that one in the last paragraph of the proof of Theorem \ref{Theorem Prescribed}, as it was described in Section \ref{Section YI}. \end{proof}

\subsection{Kodaira dimension two} LeBrun computed the Yamabe invariant of a surface of general type \cite[Theorem 7]{[LeB0]} in terms of the square of its first Chern class, and he studied some of their minimal volumes and their collapse with bounded curvatures \cite[Theorem 8]{[LeB0]}. The basic ingredients of his proofs are the non-vanishing of certain Seiberg-Witten invariants and the canonical line bundle of the complex surface. Since a symplectic 4-manifold has a canonical line bundle as well, and the symplectic 4-manifolds studied in this section have non-vanishing Seiberg-Witten invariant, his results have broader generality. In particular, he has shown the following proposition \cite{[LeB0]}. 

\begin{proposition} LeBrun. Every symplectic 4-manifold $(M, \omega)$ with $kod(M, \omega) = 2$ satisfies\begin{equation} \mathcal{Y}(M)  < 0 \end{equation} and \begin{equation} Vol_{Scal}(M) > 0. \end{equation} 
In particular, such 4-manifolds do not admit $\mathcal{F}$-structures. 
\end{proposition}

Let $\mathcal{M}_{Scal}$ be the set of Riemannian metrics on a closed 4-manifold that satisfy $Scal \geq - 12$. LeBrun defined the scalar minimal volume of $(M, g)$ as \begin{equation}Vol_{Scal}(M): = \underset{g \in \mathcal{M}_{Scal}}{inf} \int_M d vol_g.\end{equation}

\begin{proof} The claim regarding the Yamabe invariant follows from the inequalities \begin{equation} \mathcal{Y}(M) = - \underset{\gamma}{inf}|\mathcal{Y}(M, \gamma)| = -\underset{g}{inf}\left[\int_M Scal_g^2 dvol_g \right]^{\frac{1}{2}} \leq - 4\pi \sqrt{2c_1^2(M)} < 0 \end{equation}  that were studied by LeBrun. Such a symplectic manifold has a non-zero Seiberg-Witten invariant for every Riemannian metric \cite{[LeB0]}, and every conformal class has negative Yamabe constant. This implies the first equality from left to right, and that the Yamabe invariant of symplectic 4-manifolds with Kodaira dimension two is negative. The second inequality from right to left follows from \cite[Theorem 4]{[LeB0]}. It is here where the symplectic canonical bundle is used. Since the Yamabe invariant of a 4-manifold that admits an $\mathcal{F}$-structure is non-negative by Theorem \ref{Theorem PP}, it follows that symplectic 4-manifolds of Kodaira dimension two do not admit $\mathcal{F}$-structures. 

Regarding the positivity of the scalar minimal volume, we have the following chain of inequalities
\begin{equation} Vol_{Scal}(M) = \underset{g\in \mathcal{M}}{inf} \frac{(min Scal_g)^2}{144} \int_M dvol_g =  \underset{\gamma\in \mathcal{C}}{inf}\underset{g\in \gamma}{inf} \frac{(min Scal_g)^2}{144} \int_M dvol_g =\end{equation} \begin{equation} =  \underset{\gamma\in \mathcal{C}}{inf}\underset{g\in \gamma}{inf} \frac{1}{144} \int_M Scal_g^2 dvol_g = \underset{g\in \mathcal{M}}{inf}\frac{1}{144} \int_M Scal_g^2 dvol_g = \underset{g\in \mathcal{M}}{inf}\frac{|\mathcal{Y}(M, \gamma)|^2}{144} > 0.\end{equation} The set of conformal classes of Riemmanian metrics on $M$ is denoted by $\mathcal{C}$, and $\mathcal{M}$ denotes the set of Riemannian metrics. The first equality from left to right follows from rescaling properties. The equality going from equation (29) to equation (30) follows from \cite[Lemma 2]{[LeB0]}, since the non-vanishing of the Seiberg-Witten invariant implies that every conformal class has negative Yamabe constant, and therefore the hypothesis of the Lemma are satisfied. The last equality follows from \cite[Lemma 1]{[LeB0]}.

\end{proof}

Since the inequalities\begin{equation} Vol_{|K|}(M) \geq Vol_K(M) \geq Vol_{Ric}(M) \geq Vol_{Scal}(M), \end{equation} hold, where $K$ is the sectional curvature, and Ric is the Ricci curvature, Gromov's minimal volume $Vol_{|K|}(M)$ is positive for symplectic 4-manifolds of Kodaira dimension two.

\end{document}